\documentclass
{amsart}
\usepackage{graphicx}
\newtheorem{thm}{Theorem}[section]
\theoremstyle{definition}
\newtheorem{cor}[thm]{Corollary}
\newtheorem{prop}[thm]{Proposition}
\newtheorem{defn}[thm]{Definition}
\newtheorem{defns}[thm]{Definitions}
\newtheorem{lem}[thm]{Lemma}

\newtheorem{ex}[thm]{Example}

\numberwithin{equation}{section}
\begin{document}

\title[Fully $S$-coidempotent modules]{Fully $S$-coidempotent modules}

\author{F. Farshadifar*}
\address{\llap{*\,} (Corresponding Author) Department of Mathematics, Farhangian University, Tehran, Iran.}
\email{f.farshadifar@cfu.ac.ir}

\author%
{H. Ansari-Toroghy**}

\newcommand{\acr}{\newline\indent}

\address{\llap{**\,}Department of pure Mathematics\\
Faculty of mathematical
Sciences\\
University of Guilan\\
P. O. Box 41335-19141, Rasht, Iran}
\email{ansari@guilan.ac.ir}

\begin{abstract}
Let $R$ be a commutative ring with identity, $S$ be a multiplicatively closed subset of $R$, and $M$ be an $R$-module.
A submodule $N$ of $M$ is called \emph{coidempotent} if $N=((0:_MAnn^2_R(N))$.  Also, $M$ is called \emph{fully coidempotent} if every submodule of $M$ is coidempotent.
In this article, we introduce the concepts of $S$-coidempotent submodules and fully $S$-coidempotent $R$-modules as generalizations of coidempotent submodules and fully coidempotent $R$-modules. We explore some basic properties of these classes of $R$-modules.
\end{abstract}


\maketitle

\section{\bf Introduction}
\vskip 0.4 true cm
Throughout this paper $R$ will denote a commutative ring with
identity and $S$ will denote a multiplicatively closed subset of $R$. Also, $\Bbb Z$ will denote the ring of integers.

\begin{defns}\label{d1.1}
Let $M$ be an $R$-module.
\begin{itemize}
\item [(a)] $M$ is said to be a \emph{multiplication module} if for every submodule $N$ of $M$, there exists an ideal $I$ of $R$ such that $N=IM$ \cite{Ba81}.
\item [(b)]  $M$ is said to be a \emph{comultiplication module} if for every submodule $N$ of $M$, there exists an ideal $I$ of $R$ such that $N=(0:_MI)$. It is easy to see that $M$ is a comultiplication module if and only if $N=(0:_MAnn_R(N))$ for each submodule $N$ of $M$ \cite{AF07}.
\item [(c)] A submodule $N$ of $M$ is said to be \emph {pure} if $IN=N \cap IM$ for every ideal $I$ of $R$ \cite{AF74}.
$M$ is said to be \emph{fully pure} if every submodule of $M$ is pure \cite{AF122}.
\item [(d)] A submodule $N$ of $M$ is said to be \emph{copure} if $(N:_MI)=N+(0:_MI)$ for every ideal $I$ of $R$ \cite{AF09}.
$M$ is said to be \emph{fully copure} if every submodule of $M$ is copure \cite{AF122}.
\item [(e)]  A submodule $N$ of $M$ is said to be \emph{idempotent} if $N=(N:_RM)^2M$.  Also, $M$ is said to be \emph{fully idempotent} if every submodule of $M$ is idempotent \cite{AF122}.
\item [(f)] A submodule $N$ of $M$ is said to be \emph{coidempotent} if $N=((0:_MAnn^2_R(N))$.  Also, $M$ is said to be \emph{fully coidempotent} if every submodule of $M$ is coidempotent \cite{AF122}.
\end{itemize}
\end{defns}
Recently the notions such as of $S$-Noetherian rings, $S$-Noetherian modules, $S$-prime submodules, $S$-multiplication modules, $S$-2-absorbing submodules, $S$-second submodules, $S$-comultiplication modules, classical $S$-2-absorbing submodules,  $S$-pure submodules, $S$-copure submodules, fully $S$-idempotent module, etc.  introduced and investigated \cite{AD02, AS16, BRT18, satk19, ATUS2020, uatk20,  FF22, EUS2020, Na2020, FF2022, FF2023, FF2024}.

\begin{defns}\label{d1.2}
Let $M$ be an $R$-module.
\begin{itemize}
\item [(a)] A multiplicatively closed subset $S$ of $R$ is said to satisfy the \textit{maximal multiple condition} if there
exists an $s \in S$ such that $t\mid s$ for each $t \in S$. For example, if $S$ is finite or $ S \subseteq U(R)$, then $S$ satisfying the maximal multiple condition \cite{ATUS2020}.
\item [(b)] A submodule $N$ of $M$ is said to be an \textit{$S$-finite submodule} if there exist a finitely generated submodule
$K$ of $M$ and $s \in S$ such that $sN \subseteq K \subseteq N$. Also, $M$ is said to be an \textit{$S$-Noetherian module} if every submodule of $M$ is $S$-finite. In particular, $R$ is said to be an $S$-Neotherian ring
if it is an $S$-Noetherian $R$-module \cite{AD02}.
\item [(c)] A submodule $N$ of $M$ is said to be an \textit{$S$-direct summand} of $M$ if there exist a submodule $K$ of $M$ and $s\in S$ such that $sM=N+K$ (d.s.). $M$ is said to be an \textit{$S$-semisimple module} if every submodule of $M$ is an $S$-direct summand of $M$ \cite{FF2023}.
\item [(d)] $M$ is said to be an \emph{$S$-comultiplication module} if for each submodule $N$ of $M$, there exist an $s \in S$ and an ideal $I$ of $R$ such that $s(0:_MI) \subseteq N \subseteq (0:_MI)$ \cite{EUS2020}.
\item [(e)] A submodule $N$ of $M$ is said to be \emph {$S$-pure} if there exists an $s \in S$ such that $s(N \cap IM) \subseteq IN$ for every ideal $I$ of $R$. Also, $M$ is said to be \emph{fully $S$-pure} if every submodule of $M$ is $S$-pure \cite{FF2022}.
\item [(f)] A submodule $N$ of $M$ is said to be \emph {$S$-copure} if there exists an $s \in S$ such that $s(N:_MI)\subseteq N+(0:_MI)$ for every ideal $I$ of $R$. Also, $M$  is said to be \emph{fully $S$-copure} if every submodule of $M$ is $S$-copure \cite{FF2023}.
\item [(g)] A submodule $N$ of $M$ is said to be an \emph{$S$-idempotent submodule} if there exists an $s \in S$ such that $sN\subseteq  (N:_RM)^2M \subseteq N$. Also, $M$ is said to a \emph{fully $S$-idempotent module} if every submodule of $M$ is an $S$-idempotent submodule \cite{FF2024}.
\end{itemize}
\end{defns}

In this paper, we introduce the notions of $S$-coidempotent submodules and fully $S$-coidempotent modules as a generalization of coidempotent submodules and fully coidempotent modules.
Also, these notions can be regarded as a dual notions of $S$-idempotent submodules and fully $S$-idempotent modules. We consider various fundamental properties of fully $S$-coidempotent $R$-modules.
\section{\bf Main results}
\begin{defn}\label{d2.1}
We say that a submodule $N$ of an $R$-module $M$ is an \emph{$S$-coidempotent submodule} if there exists an $s \in S$ such that $s(0:_MAnn^2_R(N)) \subseteq N$.
\end{defn}

\begin{defn}\label{d2.2}
We say that an $R$-module $M$ is a \emph{fully $S$-coidempotent module} if every submodule of $M$ is an $S$-coidempotent submodule.
\end{defn}

\begin{ex}\label{e2.6}
 Let $M$ be an $R$-module with $Ann_R(M) \cap S\not=\emptyset$. Then clearly, $M$ is a fully $S$-coidempotent $R$-module.
\end{ex}

\begin{prop}\label{p2.3}
Every fully coidempotent $R$-module is a fully $S$-coidempotent $R$-module. The converse is true if $S \subseteq U(R)$, where $U(R)$ is the set of units in $R$.
\end{prop}
\begin{proof}
This is clear.
\end{proof}

The following examples show that the converse of Proposition \ref{p2.3} is not true in general.
\begin{ex}\label{e2.4}
Consider the $\Bbb Z$-module $\Bbb Z$. Then for each positive integer $t$, $(0:_{\Bbb Z}Ann_{\Bbb Z}(t\Bbb Z))= \Bbb Z$ implies that each submodule of $\Bbb Z$ is not coidempotent and so $\Bbb Z$ is not a fully coidempotent $\Bbb Z$-module.
Now,
take the multiplicatively closed subset $S = \Bbb Z \setminus \{0\}$ of $\Bbb Z$. Then  for each positive integer $t$, $t(0:_{\Bbb Z}Ann_{\Bbb Z}(t\Bbb Z)^2)=t \Bbb Z$. Thus each submodule of $\Bbb Z$ is $S$-coidempotent and hence $\Bbb Z$ is a fully $S$-coidempotent $\Bbb Z$-module.
\end{ex}

\begin{ex}\label{e3.2}
Consider the $\Bbb Z$-module $M=\Bbb Z_2\oplus \Bbb Z_2$. Take the multiplicatively closed subset $S = \{2^n: n \in \Bbb N \cup \{0\}\}$ of $\Bbb Z$, where $\Bbb N$ denotes the set of positive integers. Then $M$ as a $\Bbb Z$-module
is fully $S$-coidempotent, while $M$ is not
fully coidempotent.
\end{ex}

Lemma \ref{l2.5} and Example \ref{e2.7} show that the notion of $S$-comultiplication $R$-module is a generalization of  fully $S$-coidempotent $R$-module.
\begin{lem}\label{l2.5}
Let $M$ be a fully $S$-coidempotent $R$-module. Then $M$ is an $S$-comultiplication $R$-module.
\end{lem}
\begin{proof}
Let $N$ be a submodule of $M$. There exists an $s \in S$ such that
$s(0:_MAnn^2_R(N)) \subseteq N$. This implies that
$$
 s(0:_MAnn_R(N)) \subseteq s(N:_MAnn_R(N)) \subseteq s(0:_MAnn^2_R(N))\subseteq N.
$$
\end{proof}

The following example shows that the converse of Lemma \ref{l2.5}
is not true in general.

\begin{ex}\label{e2.7}
Take the multiplicatively closed subset $S=\Bbb Z \setminus 2\Bbb Z$ of $\Bbb Z$. Then
 $\Bbb Z_{4}$ is an $S$-comultiplication
  $\Bbb Z$-module. But $\Bbb Z_{4}$ is not a fully $S$-coidempotent $\Bbb Z$-module.
\end{ex}

A proper submodule $N$ of an $R$-module $M$ is said to be \emph{completely irreducible} if $N=\bigcap _
{i \in I}N_i$, where $ \{ N_i \}_{i \in I}$ is a family of
submodules of $M$, implies that $N=N_i$ for some $i \in I$. It is
easy to see that every submodule of $M$ is an intersection of
completely irreducible submodules of $M$ \cite{FHo06}.

In the following theorem, we characterize the fully $S$-coidempotent $R$-modules, where $S$ satisfying the maximal multiple condition.
\begin{thm}\label{t2.11}
Let $S$ satisfying the maximal multiple condition and let $M$ be an $R$-module. Then the following statements are equivalent:
\begin{itemize}
 \item [(a)] $M$ is a fully $S$-coidempotent module;
 \item [(b)] Every completely irreducible submodule of $M$ is $S$-coidempotent;
 \item [(c)] For all submodules $N$ and $K$ of $M$, we have
               $s(0:_MAnn_R(N)Ann_R(K))\subseteq N+K$ for some $s \in S$.
\end{itemize}
 \end{thm}
\begin{proof}
$(a) \Rightarrow (b)$ This is clear.

$(b) \Rightarrow (a)$
As $S$ satisfying the maximal multiple condition, there exists an $s \in S$ such that for each completely irreducible submodule $L$ of $M$ we have
 $s(0:_MAnn^2_R(L))\subseteq L$.
 Now, let $N$ be a submodule of $M$ and $N=\bigcap_{i \in I}L_i$, where each $L_i$ is a completely irreducible submodule of $M$.
Then we have
\begin{align}
\nonumber
s(0:_MAnn_R(N))&=s(0:_MAnn_R(\bigcap_{i \in I}L_i))
\\ \nonumber
&\subseteq \bigcap_{i \in I}s(0:_MAnn_R(L_i))
\\ \nonumber
&\subseteq \bigcap_{i \in I} L_i=N.
\end{align}

$(a) \Rightarrow (c)$  Let $N$ and $K$ be two submodules of $M$.
Then there exists an $s \in S$ such that $s(0:_MAnn^2_R(N+K)) \subseteq N+K$. Hence we have
$$
s(0:_MAnn_R(N)Ann_R(K))\subseteq s(0:_MAnn^2_R(N+K)) \subseteq N+K.
$$

$(c) \Rightarrow (a)$  For a submodule $N$ of $M$ for some $s \in S$,
we have
$$
s(0:_MAnn^2_R(N))=s(0:_MAnn_R(N)Ann_R(N))\subseteq N+N=N.
$$
Thus $M$ is a fully $S$-coidempotent module
\end{proof}

\begin{thm}\label{t2.7}
Let $S$ satisfying the maximal multiple condition and $M$ be an $R$-module. Then we have the following.
\begin{itemize}
\item [(a)]  If $M$ is an $S$-comultiplication module such
that every completely irreducible submodule of $M$ is an $S$-direct
summand of $M$, then $M$ is a fully $S$-coidempotent module.
\item [(b)]  If $M$ is an $S$-semisimple $S$-comultiplication
module, then $M$ is a fully $S$-coidempotent module.
\end{itemize}
\end{thm}
\begin{proof}
(a) By Theorem \ref{t2.11}, it is enough to show that every completely
irreducible submodule of $M$ is $S$-coidempotent. So, let $L$ be a
completely irreducible submodule of $M$. By hypothesis, there exists an $s \in S$ such that $sM=L+K$
(d.s.), where $K$ is a submodule of $M$. Thus
$$
s(0:_MAnn^2_R(L)) \subseteq (0:_{sM}Ann^2_R(L))=(0:_LAnn^2_R(L))+(0:_KAnn^2_R(L))=L.
$$

(b) This follows from part (a).
\end{proof}

The saturation $S^*$ of $S$ is defined as $S^*=\{x \in R : x/1\  is \ a\ unit  \ of\ S^{-1}R \}$. It is obvious that $S^*$ is a multiplicatively closed subset of $R$ containing $S$ \cite{Gr92}.
\begin{prop}\label{p2.8}
Let $M$ be an R-module. Then we have the following.
\begin{itemize}
\item [(a)] If $S_1 \subseteq S_2$ are multiplicatively closed subsets of $R$ and $M$ is a fully $S_1$-coidempotent $R$-module, then $M$ is a fully $S_2$-coidempotent $R$-module.
\item [(b)] $M$ is a fully $S$-coidempotent $R$-module if and only if $M$ is a fully $S^*$-coidempotent $R$-module.
\item [(c)] If $M$ is a fully $S$-coidempotent $R$-module, then every homomorphic image of $M$ is a fully $S$-coidempotent  $R$-module.
\end{itemize}
\end{prop}
\begin{proof}
(a) This is clear.

(b) Let $M$ be a fully $S$-coidempotent $R$-module. Since $S\subseteq S^*$,  $M$ is a fully $S^*$-coidempotent $R$-module by part (a).
For the converse, assume that $M$ is a fully $S^*$-coidempotent module and $N$ is a submodule of $M$.
Then there exists an $x \in S^*$ such that $x(0:_MAnn^2_R(N))\subseteq N$. As $x \in S^*$,  $x/1$ is a unit of $S^{-1}R$ and so $(x/1)(a/s)=1$ for some $a \in R$ and $s \in S$. This implies that $us = uxa$ for some $u \in S$. Thus we have
$$
us(0:_MAnn^2_R(N)) =uxa(0:_MAnn^2_R(N)) \subseteq x(0:_MAnn^2_R(N)) \subseteq N.
$$
Therefore, $M$ is a fully $S$-coidempotent $R$-module.

(c) Let $N$ be a submodule of $M$ and $K/N$ a submodule of $M/N$. By assumption, there exists an $s \in S$ such that $s(0:_MAnn^2_R(K))\subseteq K$. Hence
\begin{align}
\nonumber
(0:_MAnn^3_R(K))&=((0:_MAnn^2_R(K)):_MAnn_R(K))
\\ \nonumber
&\subseteq (K:_MsAnn_R(K))
\\ \nonumber
&\subseteq (0:_MsAnn^2_R(K)).
\end{align}
It follows that
$s^2(0:_MAnn^3_R(K))\subseteq K$. Thus
\begin{align}
\nonumber
s^2(0:_{M/N}Ann^2(K/N))&\subseteq s^2(0:_MAnn_R(N)Ann^2_R(K/N))/N
\\ \nonumber
&\subseteq s^2(0:_MAnn^3_R(K))/N\subseteq K/N,
\end{align}
as needed.
\end{proof}

The following theorem provides some characterizations for fully coidempotent $R$-modules.
\begin{thm}\label{t2.8}
Let $M$ be an R-module. Then the following statements are equivalent:
\begin{itemize}
\item [(a)] $M$ is a fully coidempotent $R$-module;
\item [(b)] $M$ is a fully $(R-P)$-coidempotent $R$-module for each prime ideal $P$ of $R$;
\item [(c)] $M$ is a fully $(R-\mathfrak{m})$-coidempotent $R$-module for each maximal ideal $\mathfrak{m}$ of $R$;
\item [(d)] $M$ is a fully $(R-\mathfrak{m})$-coidempotent $R$-module for each maximal ideal $\mathfrak{m}$ of $R$ with $M_{\mathfrak{m}}\not=0_{\mathfrak{m}}$.
\end{itemize}
\end{thm}
\begin{proof}
$(a) \Rightarrow (b)$
 Let $M$ be a fully coidempotent $R$-module and $P$ be a prime ideal of $R$. Then $R-P$ is multiplicatively
closed set of $R$ and so $M$ is a fully $(R-P)$-coidempotent $R$-module by Proposition \ref{p2.3}.

$(b) \Rightarrow (c)$
Since every maximal ideal is a prime ideal, the result follows from the part (b).

$(c) \Rightarrow (d)$
This is clear.

$(d) \Rightarrow (a)$
Let $N$ be a submodule of $M$. Take a maximal ideal $\mathfrak{m}$ of $R$ with $M_{\mathfrak{m}}\not=0_{\mathfrak{m}}$. As
$M$ is a fully $(R-\mathfrak{m})$-coidempotent module, there exists an $s \not \in \mathfrak{m}$ such that $s(0:_MAnn^2_R(N))\subseteq N$. This implies that
$$
(0:_MAnn^2_R(N))_{\mathfrak{m}}=(s(0:_MAnn^2_R(N)))_{\mathfrak{m}}\subseteq N_{\mathfrak{m}}.
$$
If $M_{\mathfrak{m}}=0_{\mathfrak{m}}$, then clearly $(0:_MAnn^2_R(N))_{\mathfrak{m}}\subseteq N_{\mathfrak{m}}$. So, we have $(0:_MAnn^2_R(N))_{\mathfrak{m}}\subseteq N_{\mathfrak{m}}$ for each maximal ideal $\mathfrak{m}$ of $R$. It follows that $(0:_MAnn^2_R(N))\subseteq N$. Thus $N=(0:_MAnn^2_R(N))$ because the inverse inclusion is clear.
\end{proof}

\begin{prop}\label{p2.15}
Let $f : M \rightarrow \acute{M}$ be an $R$-monomorphism of $R$-modules. Then we have the following.
\begin{itemize}
\item [(a)] If $\acute{M}$ is a fully $S$-coidempotent module, then $M$ is a fully $S$-coidempotent module.
\item [(b)] If $M$ is a fully $S$-coidempotent module and $t\acute{M} \subseteq f(M)$ for some
$t \in S$, then $\acute{M}$ is a fully $S$-coidempotent module.
\end{itemize}
\end{prop}
\begin{proof}
(a) Let $N$ be a submodule of $M$. Then as $\acute{M}$ is a fully  $S$-coidempotent module, there exists an $s \in S$ such that $s(0:_{\acute{M}}Ann^2_R(f(N))) \subseteq f(N)$. This implies that  $sf^{-1}((0:_{\acute{M}}Ann^2_R(f(N))) )\subseteq f^{-1}(f(N))$.
It follows that $s(0:_MAnn^2_R(N)) \subseteq N$.

(b) Let $\acute{N}$ be a submodule of $\acute{M}$ and $t\acute{M} \subseteq f(M)$ for some
$t \in S$.  Since $M$ is
a fully $S$-coidempotent module, there exists an $s\in S$ such that $s(0:_MAnn^2_R(f^{-1}(\acute{N}))) \subseteq f^{-1}(\acute{N})$. Now as $f$ is monomorphism, we have that
$$
s(0:_{f(M)}Ann^2_R(f^{-1}(\acute{N}))) \subseteq f(f^{-1}(\acute{N}))=f(M) \cap \acute{N} \subseteq \acute{N}.
$$
One can see that  $Ann^2_R(f^{-1}(\acute{N}))=Ann^2_R(\acute{N})$. Therefore, $s(0:_{f(M)}Ann^2_R(\acute{N})) \subseteq \acute{N}$. Thus $ts(0:_{\acute{M}}Ann^2_R(\acute{N})) \subseteq \acute{N}$, as needed.
\end{proof}

\begin{cor}\label{c2.15}
Every submodule of a fully $S$-coidempotent $R$-module is a fully $S$-coidempotent $R$-module.
\end{cor}
\begin{proof}
Let $N$ be a submodule of a  fully $S$-coidempotent $R$-module $M$. Then the result follows from Proposition \ref{p2.15} by using the inclusion homomorphism $f : N \hookrightarrow M$.
\end{proof}

Let $R_i$ be a commutative ring with identity, $M_i$ be an $R_i$-module for each $i = 1, 2,..., n$, and $n \in \Bbb N$. Assume that
$M = M_1\times M_2\times ...\times M_n$ and $R = R_1\times R_2\times ...\times R_n$. Then $M$ is
an $R$-module with componentwise addition and scalar multiplication. Also,
if $S_i$ is a multiplicatively closed subset of $R_i$ for each $i = 1, 2,...,n$,  then
$S = S_1\times S_2\times ...\times S_n$ is a multiplicatively closed subset of $R$. Furthermore,
each submodule $N$ of $M$ is of the form $N = N_1\times N_2\times...\times N_n$, where $N_i$ is a
submodule of $M_i$.
\begin{thm}\label{l2.12}
Let $M_i$ be an $R_i$-module and $S_i \subseteq R_i$ be a multiplicatively closed subset for $i=1,2$. Assume that
$M = M_1\times  M_2$, $R = R_1\times  R_2$, and $S = S_1\times  S_2$. Then $M$ is a fully $S$-coidempotent $R$-module if and only
if $M_1$ is a fully $S_1$-coidempotent $R_1$-module and $M_2$ is a fully $S_2$-coidempotent $R_2$-module.
\end{thm}
\begin{proof}
 First assume that $M$ is a fully $S$-coidempotent $R$-module, without loss of generality we show that $M_1$ is a fully
 $S_1$-coidempotent $R_1$-module.
Let $N_1$ be a submodule of $M_1$. Then $N_1\times  \{0\}$ is a submodule of $M$. As $M$ is a fully $S$-coidempotent
$R$-module, there exists an $s = (s_1, s_2)  \in S_1\times  S_2$ such that $(s_1, s_2)(0:_MAnn^2_{R_1}(N_1\times \{0\})) \subseteq N_1\times  \{0\}$. This in turn implies that $s_1(0:_{M_1}Ann^2_{R_1}(N_1)) \subseteq N_1$. Thus $M_1$ is a fully $S_1$-coidempotent $R_1$-module. Now assume that $M_1$ is a fully $S_1$-coidempotent $R_1$-module and $M_2$ is a fully $S_2$-coidempotent $R_2$-module. Let $N$ be a submodule of $M$. Then $N$ must be in the form of $N_1\times  N_2$, where $N_1$ is a submodule of $M_1$ and $N_2$ is a submodule of $M_2$. As $M_1$ is a fully $S_1$-coidempotent $R_1$-module, there exists an $s_1\in S_1$
such that  $s_1(0:_{M_1}Ann^2_{R_1}(N_1)) \subseteq N_1$. Similarly, there exists an element $s_2 \in S_2$ such that  $s_2(0:_{M_2}Ann^2_{R_2}(N_2)) \subseteq N_2$. Set $s = (s_1, s_2) \in S$. Then we have
\begin{align}
 \nonumber
(s_1, s_2)(0:_MAnn^2_R(N)) &\subseteq s_1(0:_{M_1}Ann^2_{R_1}(N_1))\times  s_2(0:_{M_2}Ann^2_{R_2}(N_2))
\\ \nonumber
&  \subseteq N_1 \times N_2=N.
\end{align}
So, $M$ is a fully $S$-coidempotent $R$-module.
\end{proof}

\begin{thm}\label{t2.13}
Let $M_i$ be an $R_i$-module and $S_i$  be a multiplicatively closed subset of $R_i$ for $i=1,2,..,n$.  Assume that $M = M_1\times ...\times M_n$, $R = R_1\times ...\times R_n$ and
$S = S_1\times ...\times S_n$. Then the following statements are equivalent:
\begin{itemize}
\item [(a)] $M$ is a fully $S$-coidempotent $R$-module;
\item [(b)] $M_i$ is a fully $S_i$-coidempotent $R_i$-module for each $i\in \{1, 2, ..., n\}$.
\end{itemize}
\end{thm}
\begin{proof}
We use induction. If $n =1$, the claim is trivial. If $n =2$, the
claim follows from Theorem \ref{l2.12}. Assume that the claim is true for $n<k$ and we
show that it is also true for $n=k$. Put $M = (M_1\times ... \times M_{n-1})\times  M_n$, $R = (R_1\times  R_2\times ...\times  R_{n-1})\times  R_n$ and $S = (S_1\times ...\times S_{n-1})\times  S_n$. By Theorem \ref{l2.12}, $M$ is fully $S$-coidempotent $R$-module if and only if $M_1\times ... \times M_{n-1}$ is a fully $(S_1\times ...\times S_{n-1})$-coidempotent  $(R_1\times  R_2\times ...\times  R_{n-1})$-module and $M_n$ is a fully $S_n$-coidempotent $R_n$-module. Now the rest follows from the induction hypothesis.
\end{proof}

The following lemma is known, but we write its proof here for the sake of references.
\begin{lem}\label{l33.10}
Let $M$ be an $R$-module. Then we have the following.
\begin{itemize}
\item [(a)] If $S$ satisfying the maximal multiple condition, then $S^{-1}(0:_MI)= (0:_{S^{-1}M}S^{-1}I)$ and $S^{-1}(Ann_R(N))=Ann_{S^{-1}R}(S^{-1}N)$ for each ideal $I$ of $R$ and submodule $N$ of $M$.
\item [(b)] If $I$ is an $S$-finite ideal of $R$ and $N$ is an $S$-finite submodule of $M$, then $S^{-1}(0:_MI)= (0:_{S^{-1}M}S^{-1}I)$ and $S^{-1}(Ann_R(N))=Ann_{S^{-1}R}(S^{-1}N)$.
\end{itemize}
 \end{lem}
\begin{proof}
(a) Let $x/h \in (0:_{S^{-1}M}S^{-1}I)$ and $a \in I$. Then $(x/h)(a/1)=0$. Hence there exists an $u \in S$ such that $uxa=0$. As $S$ satisfying the maximal multiple condition, there exists an $s \in S$ such that $t\mid s$ for each $t \in S$. Thus $s=uv$ for some $v \in S$. Therefore, $sxI=0$. So, $sx \in (0:_MI)$. Hence, $x/h=(xs)/(sh) \in S^{-1}(0:_MI)$ and so  $(0:_{S^{-1}M}S^{-1}I)\subseteq S^{-1}(0:_MI)$. The reverse inclusion is clear.  Similarly, one can see that $S^{-1}(Ann_R(N))=Ann_{S^{-1}R}(S^{-1}N)$.

(b) Let $I$ be an $S$-finite ideal of $R$. Then there exist $t \in S$ and $a_1,a_2,...,a_n \in I$ such that
$tI \subseteq Ra_1+...+Ra_n \subseteq I$. Now let $x/h \in (0:_{S^{-1}M}S^{-1}I)$. Then $(x/h)(a_i/1)=0$ for $i=1,2,..,n$. Hence there exists an $u_i \in S$ such that $u_ixa_i=0$ for $i=1,2,..,n$. Set $U=u_1u_2...u_n$. Then we have $uxa_i=0$ for $i=1,2,..,n$. This implies that $utxI=0$. Hence,
 $utx \in (0:_MI)$. Thus $x/h=(xut)/(uth) \in S^{-1}(0:_MI)$ and so  $(0:_{S^{-1}M}S^{-1}I)\subseteq S^{-1}(0:_MI)$. The reverse inclusion is clear.   Similarly, if $N$ is an $S$-finite submodule of $M$ one can see that $S^{-1}(Ann_R(N))=Ann_{S^{-1}R}(S^{-1}N)$.
\end{proof}

\begin{thm}\label{p33.10}
Let $S$ satisfying the maximal multiple condition and $M$ be an $R$-module. Then $M$ is a fully $S$-coidempotent $R$-module if and only if $S^{-1}M$ is a fully coidempotent $S^{-1}R$-module.
 \end{thm}
\begin{proof}
First note that as $S$ satisfying the maximal multiple condition, $S^{-1}(0:_MAnn^2_R(N))=(0:_{S^{-1}M}Ann^2_{S^{-1}R}(S^{-1}N))$ for each submodule $N$ of $M$ by Lemma \ref{l33.10} (a). Now suppose that $M$ is a fully $S$-coidempotent $R$-module and $S^{-1}N$ is a submodule of $S^{-1}M$. Then there exists an $t \in S$ such that $t(0:_MAnn^2_R(N)) \subseteq N$. Let
$x/h_1 \in(0:_{S^{-1}M}Ann^2_{S^{-1}R}(S^{-1}N))$. Then $x/h_1 \in S^{-1}(0:_MAnn^2_R(N))$ and so  $x/h_1=y/h_2$, where $y \in (0:_MAnn^2_R(N))$ and $h_2 \in S$. Thus there exists $u \in S$ such that $xh_2u=yh_1u$. So, $txh_2u=yh_1ut \in N$. It follows that $x/h_1=(xth_2u)/(h_1th_2u) \in S^{-1}N$.
This in turn implies that $S^{-1}N=(0:_{S^{-1}M}Ann^2_{S^{-1}R}(S^{-1}N))$ and $S^{-1}M$ is a fully coidempotent $S^{-1}R$-module. Conversely, assume that $S^{-1}M$ is a fully coidempotent $S^{-1}R$-module and $N$ is a submodule of $M$. Then $S^{-1}N=(0:_{S^{-1}M}Ann^2_{S^{-1}R}(S^{-1}N))$. Let $m \in (0:_MAnn^2_R(N))$. Then $m/1=n/s_1$ for some $n \in N$ and $s_1 \in S$. Thus $ms_1s_2=ns_2$ for some $s_2 \in S$. As $S$ satisfying the maximal multiple condition, there exists an $s \in S$ such that  $t\mid s$ for each $t \in S$. Therefore, $ms^2=nsv \in N$ for some $v \in S$. Thus $x \in (N:_Ms^2)$. Thus $s^2(0:_MAnn^2_R(N)) \subseteq N$, as needed.
\end{proof}

\begin{prop}\label{p333.10}
Let $R$ be an $S$-Noetherian ring. If $M$ is a fully $S$-coidempotent $R$-module, then $S^{-1}M$ is a fully coidempotent $S^{-1}R$-module.
 \end{prop}
\begin{proof}
This is straightforward by using the technique of Theorem \ref{p33.10} and Lemma \ref{l33.10} (b).
\end{proof}

\begin{thm} \label{t3.3}
Let $M$ be an $S$-comultiplication $R$-module and $N$ be a submodule of $M$.
Then the following statements are equivalent:
\begin {itemize}
\item [(a)] $N$ is an $S$-copure submodule of $M$;
\item [(b)] $M/N$ is an $S$-comultiplication $R$-module and
$N$ is an $S$-coidempotent submodule of $M$;
\item [(c)] $M/N$ is an $S$-comultiplication $R$-module and there exists an $s \in S$ such that
$s(N:_MAnn_R(K))\subseteq K$, where $K$ is a submodule of $M$ with $N
\subseteq K$;
\item [(d)] $M/N$ is an $S$-comultiplication $R$-module and  there exists an $s \in S$ such that
$s(N:_MAnn_R(K))\subseteq (N:_M(N:_RK))$, where $K$ is a submodule of $M$.
\end {itemize}
\end{thm}
\begin{proof}
$(a) \Rightarrow (b)$
Assume that $K/N$ is a submodule of $M/N$.
Since $N$ is an $S$-copure submodule of $M$, there exists an $t \in S$ such that $t(N:_MAnn_R(K/N)) \subseteq N+(0:_MAnn_R(K/N))$. Also, as $M$ is an $S$-comultiplication module, there exists an $s \in S$ such that $s(0:_MAnn_R(K)) \subseteq K$. Now we have
\begin{align}
\nonumber
st(0:_{M/N}Ann_R(K/N))&= st((N:_MAnn_R(K/N))/N)
\\ \nonumber
&=s((t(N:_MAnn_R(K/N))+N)/N)
\\ \nonumber
&\subseteq s(N+(0:_MAnn_R(K/N))/N)
\\ \nonumber
&\subseteq(N+s(0:_MAnn_R(K))/N
\\ \nonumber
&\subseteq (N+K)/N=K/N.
\end{align}
Thus $M/N$ is an $S$-comultiplication $R$-module. Now we show that
$N$ is an $S$-coidempotent submodule of $M$. As $M$ is an $S$-comultiplication module, there exists an $u \in S$ such that $u(0:_MAnn_R(N)) \subseteq N$. As $N$ is an $S$-copure submodule of $M$, there exists an $h \in S$ such that $h(N:_MuAnn_R(N)) \subseteq N+(0:_MuAnn_R(N))$. Now we have
\begin{align}
\nonumber
h(0:_MAnn^2_R(N)) &\subseteq h((N:_Mu):_MAnn_R(N))
\\ \nonumber
&=h(N:_MuAnn_R(N))
\\ \nonumber
&\subseteq N+(0:_MuAnn_R(N))
\\ \nonumber
& =((0:_MAn_R(N):_Mu).
\end{align}
This implies that $uh(0:_MAnn^2_R(N))\subseteq (0:_MAn_R(N)$.
Thus $u^2h(0:_MAnn^2_R(N))\subseteq u(0:_MAn_R(N)\subseteq N$, as desired.

$(b) \Rightarrow (c)$
Let $K$ be a submodule of $M$ with $N
\subseteq K$. Since $M/N$ is an $S$-comultiplication $R$-module, there exists an $t \in S$ such that $t(0:_{M/N}Ann_R(K/N))\subseteq K/N$. It follows that $t(N:_MAnn_R(K/N))+N \subseteq K$ and so $t(N:_MAnn_R(K/N)) \subseteq K$. As $N$ is an $S$-coidempotent, there exists an $s \in S$ such that $s(0:_MAnn^2_R(N)) \subseteq N$. Now we have
\begin{align}
\nonumber
(N:_MAnn_R(K))/N &\subseteq (N:_MAnn_R(N)Ann_R(K/N))/N
\\ \nonumber
 &\subseteq ((0:_MAnn^2_R(N)):_MAnn_R(K/N))N
 \\ \nonumber
 &\subseteq ((N:_MsAnn_R(K/N))/N.
\end{align}
It follows that
$s(N:_MAnn_R(K))\subseteq (N:_MAnn_R(K/N))$.
Thus
$$
st(N:_MAnn_R(K))\subseteq t(N:_MAnn_R(K/N))\subseteq K.
$$

$(c) \Rightarrow (a)$
Let $I$ be an ideal of $R$. Since $N
\subseteq (0:_MI)+N$, there exists an $s \in S$ such that $s(N:_MAnn_R((0:_MI)+N)) \subseteq (0:_MI)+N$ by part (c).
As $M$ is an $S$-comultiplication module, there exists an $t \in S$ such that $t(0:_MAnn_R(N)) \subseteq N$.
Now we have
\begin{align}
\nonumber
(N:_MI)\subseteq ((0:_MAnn_R(N)):_MI)&=((0:_MI):_MAnn_R(N))
\\ \nonumber
&\subseteq (N+(0:_MI):_MAnn_R(N))
\\ \nonumber
&\subseteq ((0:_MAnn_R(N+(0:_MI)):_MAnn_R(N))
\\ \nonumber
 &\subseteq (N:_MtAnn_R(N+(0:_MI))).
\end{align}

This implies that $t(N:_MI) \subseteq (N:_MAnn_R(N+(0:_MI)))$. Thus
$$
st(N:_MI) \subseteq s(N:_MAnn_R(N+(0:_MI)))\subseteq N+(0:_MI).
$$

$(b) \Rightarrow (d)$
Let $K$ be a submodule of $M$. As  $N$
is $S$-coidempotent, there exists an $s \in S$ such that $s(0:_MAnn^2_R(N)) \subseteq N$. Now we have
\begin{align}
\nonumber
(N:_MAnn_R(K)) &\subseteq (N:_MAnn_R(N)(N:_RK))
\\ \nonumber
&=((N:_MAnn_R(N)):_M(N:_RK))
\\ \nonumber
&\subseteq ((0:_MAnn^2_R(N)):_M(N:_RK))
\\ \nonumber
&\subseteq (N:_Ms(N:_RK)).
\end{align}
This implies that $s(N:_MAnn_R(K))\subseteq (N:_M(N:_RK))$.

$(d) \Rightarrow (b)$
Take $K=N$.
\end{proof}

\begin{cor} \label{c3.4}
 Let $M$ be an $R$-module. Then we have the following.
\begin{itemize}
\item [(a)] If $M$ is a fully $S$-coidempotent module,
  then $M$ is fully $S$-copure.
\item [(b)] If $M$ is an $S$-comultiplication fully $S$-copure module,
  then $M$ is fully $S$-coidempotent.
    \end{itemize}
  \end{cor}
\begin{proof}
(a) By  Proposition \ref{p2.8},  every homomorphic image of
$M$ is a fully $S$-coidempotent module and so every homomorphic image of
$M$ is an $S$-comultiplication module by Lemma \ref{l2.5}.
Now the result follows from Theorem \ref{t3.3} $(b) \Rightarrow (a)$.

(b) This follows from Theorem \ref{t3.3} $(a) \Rightarrow (b)$.
\end{proof}

The following example shows that in part (b) of
the Corollary \ref{c3.4}, the condition $M$ is an $S$-comultiplication $R$-module can not be omitted.
\begin{ex}\label{e3.2}
Set $M=\Bbb Z_p\oplus \Bbb Z_p$ for some prime number $p$. Take the multiplicatively closed subset $S= \Bbb Z \setminus p\Bbb Z$. Then $M$ as a $\Bbb Z$-module
is a fully $S$-copure module, while $M$ is not a
fully $S$-coidempotent module because the submodule $0\oplus \Bbb Z_p$ of $M$ is not $S$-coidempotent.
\end{ex}

\begin{prop}\label{p3.10}
Let $S$ satisfying the maximal multiple condition and $M$ be a fully $S$-coidempotent $R$-module. Then for each submodule $K$ of $M$ and each collection $\{N_\lambda
\}_{\lambda \in \Lambda}$ of submodules of $M$,  there exists an $s \in S$ such that $
s(\bigcap _{\lambda \in \Lambda}(N_\lambda +K)) \subseteq  \bigcap_{\lambda
\in \Lambda} N_\lambda +K.$
 \end{prop}
\begin{proof}
Let $K$ be a submodule of $M$ and $\{N_\lambda
\}_{\lambda \in \Lambda}$ be a collection of submodules of $M$.
By  Lemma \ref{l2.5}, $M$ is an $S$-comultiplication $R$-module. So, there exists an $s \in S$ such that $s(0:_MAnn_R(N_\lambda) \subseteq N_\lambda$ for each $\lambda \in \Lambda$ since $S$ satisfying the maximal multiple condition. Now by using the fact that $M$ is fully $S$-copure by Corollary \ref{c3.4} (a), we have
\begin{align}
\nonumber
s^3(\bigcap _{\lambda \in \Lambda}(N_\lambda +K)) &\subseteq
\bigcap _{\lambda \in \Lambda}s^3(N_\lambda +K)
\\
\nonumber
&\subseteq
\bigcap _{\lambda \in \Lambda}s^3((0:_MAnn_R(N_\lambda))+(0:_MAnn_R(K))
\\
\nonumber
&\subseteq
\bigcap _{\lambda \in \Lambda}s^3((0:_MAnn_R(N_\lambda)Ann_R(K))
\subseteq
\bigcap _{\lambda \in \Lambda}s^3(N_\lambda:_MsAnn_R(K))
\\
\nonumber
&\subseteq \bigcap _{\lambda \in \Lambda}s^2(N_\lambda+(0:_MsAnn_R(K)) \subseteq  \bigcap _{\lambda \in \Lambda}s^2(N_\lambda:_MAnn_R(K))
\\
\nonumber
&\subseteq \bigcap _{\lambda \in \Lambda}s(N_\lambda+(0:_MAnn_R(K)) \subseteq
\bigcap _{\lambda \in \Lambda}(sN_\lambda+s(0:_MAnn_R(K))
\\
\nonumber
&\subseteq
\bigcap _{\lambda \in \Lambda}(N_\lambda+K).
\end{align}
\end{proof}

\begin{prop} \label{p3.6}
Let $M$ be an $S$-comultiplication $R$-module and $N$ be an $S$-pure submodule of $M$. Then $N$ is an $S$-coidempotent submodule of $M$.
\end{prop}
\begin{proof}
As $N$ is an $S$-pure submodule of $M$, there exists an $s \in S$ such that
$$
s(N\cap Ann_R(N)M) \subseteq Ann_R(N)N=0.
$$
Hence,  $N\cap Ann_R(N)M)\subseteq (0:_Ms)$.
Since $M$ is an $S$-comultiplication module, there exists an $t \in S$ such that $t(0:_MAnn_R(N)) \subseteq N$. Hence we have
\begin{align}
\nonumber
(0:_MAnn^2_R(N))&=((0:_MAnn_R(N)):_MAnn_R(N))
\\ \nonumber
&\subseteq ((N:_Mt):_MAnn_R(N))
\\ \nonumber
&=(N:_MtAnn_R(N))
\\ \nonumber
&=(N \cap Ann_R(N)M:_MtAnn_R(N))
\\
\nonumber
&\subseteq ((0:_Ms):_MtAnn_R(N))
\\ \nonumber
&=((0:_MAnn_R(N)):_Mst) \subseteq (N:_Mst^2).
\end{align}
Therefore, $st^2(0:_MAnn^2_R(N)) \subseteq N$.
\end{proof}

An $R$-module $M$ is said to be an \textit{$S$-multiplication module} if for each submodule $N$ of $M$, there exist $s \in S$ and an ideal $I$ of $R$ such that $sN \subseteq IM \subseteq N$ \cite{ATUS2020}.
\begin{cor} \label{t3.7}
 Let $M$ be an $R$-module. Then we have the
following.
\begin{itemize}
 \item [(a)] If $M$ is an $S$-multiplication fully $S$-copure
  module, then $M$ is fully $S$-pure.
  \item [(b)] If $M$ is an $S$-comultiplication
  fully $S$-pure module, then $M$ is fully $S$-copure.
  \item [(c)] If $M$ is an $S$-multiplication fully $S$-coidempotent
  module, then $M$ is fully $S$-idempotent.
  \item [(d)] If $M$ is  an $S$-comultiplication
  fully $S$-idempotent module, then $M$ is fully $S$-coidempotent.
\end{itemize}
\end{cor}
\begin{proof}
(a) By \cite[Proposition 2.23]{FF2024}, every submodule of $M$ is $S$-idempotent.
Hence the result follows from \cite[Corollary 2.22]{FF2024}.

(b) By Proposition \ref{p3.6}, every submodule of $M$ is $S$-coidempotent.
Hence the result follows from Corollary \ref{c3.4} (a).

(c) This follows from Corollary \ref{c3.4} (a) and \cite[Proposition 2.23]{FF2024}.

(d) By \cite[Corollary 2.22]{FF2024}, $M$ is fully $S$-pure. Thus by part (b), $M$ is fully $S$-copure.
So the result follows from Corollary \ref{c3.4} (b).
\end{proof}

\begin{lem}\label{l22.5}
Let $M$ be an $S$-comultiplication $R$-module. If $I$ and $J$ are ideals of $R$ such that $(0:_MI) \subseteq (0:_MJ)$. Then there exists an $s \in S$ such that $sJM \subseteq IM$.
\end{lem}
 \begin{proof}
As $(0:_MI) \subseteq (0:_MJ)$, we have
$$
Ann_R(IM)=((0:_MI):_RM) \subseteq ((0:_MJ):_RM)=Ann_R(JM).
$$
Now by \cite[Lemma 1]{EUS2020}, there exists an $s \in S$ such that $sJM \subseteq IM$.
\end{proof}

In the following theorem we see that if $M$ is an $S$-finite $R$-module, then  $M$ is a fully $S$-idempotent module iff $M$ is a fully $S$-coidempotent module.
\begin{thm}\label{t2.5}
Let $M$ be an $R$-module. Then we have the following.
\begin{itemize}
\item [(a)] If $M$ is an $S$-finite fully $S$-coidempotent $R$-module,  then
  $M$ is a fully $S$-idempotent module.
\item [(b)] If $M$ is an $S$-Noetherian fully $S$-idempotent module,
then $M$ is a fully $S$-coidempotent module.
\end{itemize}
\end{thm}
 \begin{proof}
(a) Let $N$ be a submodule of $M$. Since $M$ is a fully $S$-coidempotent module, there exists an $s \in S$ such that
$s(0:_MAnn^2_R(N)) \subseteq N$. Hence
\begin{align}
\nonumber
s(0:_{M/N}Ann_R(N))&=s((N:_MAnn_R(N))/N)
\\ \nonumber
&\subseteq s((0:_MAnn_R(N)):_MAnn_R(N))/N)
\\ \nonumber
&=s((0:_MAnn^2_R(N))/N)
\\ \nonumber
&=(s(0:_MAnn^2_R(N))+N)/N=0.
\end{align}
This implies that $(0:_{M/N}Ann_R(N)) \subseteq (0:_{M/N}s)$.
By Proposition \ref{p2.8} (c), $M/N$ is an $S$-coidempotent $R$-module. Thus $M/N$ is an $S$-comultiplication module by Lemma \ref{l2.5}. Hence by Lemma \ref{l22.5}, there exists $t \in S$ such that $ts(M/N) \subseteq Ann_R(N)(M/N)$. As $M$ is an $S$-finite $R$-module, one can see that $M/N$ is an $S$-finite $R$-module. Therefore, there exist $h \in S$ and $a \in Ann_R(N)$ such that $(h+a)(M/N)=0$ by \cite[Lemma 2.1]{Ha20}. So, $h+a \in (N:_RM)$. Hence,  $(h+a)^2 \in (N:_RM)^2$. This implies that $Rh^2 \subseteq Ann_R(N)+(N:_RM)^2$. Thus we have
\begin{align}
\nonumber
h^2N= Rh^2N&\subseteq Ann_R(N)N +(N:_RM)^2N
\\ \nonumber
&\subseteq (N:_RM)^2N
\\ \nonumber
&\subseteq (N:_RM)^2M.
\end{align}
Therefore,  $M$ is fully $S$-idempotent.

(b) Let $N$ be a submodule of $M$. Since $M$ is fully $S$-idempotent, there exists an $s \in S$ such that
$sN\subseteq (N:_RM)^2M\subseteq (N:_RM)N$.  Since $M$ is $S$-Noetherian, $N$ is $S$-finite.
Now by \cite[Lemma 2.1]{Ha20}, there exist $h \in S$ and $a \in (N:_RM)$ such that $(h+a)N=0$. So, $h+a \in Ann_R(N)$. Hence, $(h+a)^2 \in Ann^2_R(N)$. This yields that $Rh^2 \subseteq (N:_RM)+Ann^2_R(N)$. Thus
\begin{align}
\nonumber
h^2(0:_MAnn^2_R(N))&= Rh^2(0:_MAnn^2_R(N))
\\ \nonumber
&\subseteq (N:_RM)(0:_MAnn^2_R(N))+Ann^2_R(N)(0:_MAnn^2_R(N))
\\ \nonumber
&=(N:_RM)(0:_MAnn^2_R(N))
\\ \nonumber
&\subseteq (N:_RM)M \subseteq N,
\end{align}
as needed.
\end{proof}

\end{document}